\documentclass[graybox]{svmult}

\usepackage[T1]{fontenc}
\usepackage{mathptmx}       % selects Times Roman as basic font
\usepackage{helvet}         % selects Helvetica as sans-serif font
\usepackage{courier}        % selects Courier as typewriter font
\usepackage{type1cm}        % activate if the above 3 fonts are
\usepackage{makeidx}         % allows index generation
\usepackage{graphicx}        % standard LaTeX graphics tool
\usepackage{multicol}        % used for the two-column index
\usepackage[bottom]{footmisc}% places footnotes at page bottom
\usepackage{amsmath, amsfonts, amssymb}
\usepackage{color}
\usepackage{tikz}

\usepackage{stmaryrd}

\makeindex             % used for the subject index
\def\ds{\displaystyle}
\def\pa{\partial}
\def\Div{\mathrm{div}}

\def\N{\mathbb N}
\def\R{\mathbb R}
\def\T{\mathcal T}

\def\dsig{{\rm d}_{\sigma}}

\def\E{\mathcal E}

\def\Om{\Omega}
\def\Linf{L^{\infty}}

\begin{document}

\title*{$L^\infty$ bounds for numerical solutions of noncoercive convection-diffusion equations}
\titlerunning{$L^\infty$ bounds for convection-diffusion equations}
\author{Claire Chainais-Hillairet and Maxime Herda}

\institute{
Claire Chainais-Hillairet  \at  Univ. Lille, CNRS,UMR 8524, Inria-Laboratoire Paul Painlev\'e. F-59000 Lille, France. \\ \email{claire.chainais@univ-lille.fr, }
\and
Maxime Herda
\at Inria, Univ. Lille, CNRS, UMR 8524 - Laboratoire Paul Painlev\'e, F-59000 Lille, France.\\ \email{maxime.herda@inria.fr}
}
\maketitle

\abstract{In this work, we apply an iterative energy method \emph{\`a la} de Giorgi in order to establish $L^{\infty}$ bounds for numerical solutions of noncoercive convection-diffusion equations with mixed Dirichlet-Neumann boundary conditions.}

\keywords{finite volume schemes, uniform bounds, noncoercive elliptic equations \\[1pt]
{\bf MSC }(2010){\bf:} 65M08, 35B40.
}

% ====================================
\section{Introduction}

{\bf The continuous problem.} Let $\Om$ be an open bounded polygonal domain of $\R^p$ with $p=2$ or $3$. We denote by  ${\rm m(\cdot)}$ both the Lebesgue and $p-1$ dimensional Hausdorff measure. We assume that $\pa \Omega= \Gamma^D\cup\Gamma^N$ with $\Gamma^D\cap\Gamma^N=\emptyset$ and ${\rm m}(\Gamma^D)>0$ and we denote by ${\mathbf n}$ the exterior normal to $\pa \Omega$.  Let ${\mathbf U}\in C({\bar \Om})^2$ be a velocity field, $b\in \Linf(\Omega)$ assumed to be nonnegative, $f\in L^\infty(\Omega)$ a source term and $v^D\in \Linf (\Gamma^D)$ a boundary condition. 

We consider the following convection-diffusion equation with mixed boundary conditions:
\begin{subequations}\label{pb_depart}
\begin{align}
&\Div (-\nabla v + {\mathbf U} v)+ bv=f&&\qquad \mbox{in }\Omega, \label{eq_v}\\
& (-\nabla v + {\mathbf U} v)\cdot {\mathbf n}=0 &&\qquad \mbox{on } \Gamma^N, \label{Neum_bc}\\
&v=v^D&&\qquad \mbox{on } \Gamma^D \label{Dir_bc}.
\end{align}
\end{subequations}
This noncoercive elliptic linear problem has been widely studied by Droniou and coauthors, even with less regularity on the {data}, see for instance \cite{droniou_potan_2002, DG_M2AN_2002, Droniou_jnm_2003,DGH_sinum_2003}. Nevertheless, up to our knowledge, the derivation of explicit $L^\infty$ bounds on numerical solutions has not been done in the literature.

\bigskip

\noindent
{\bf  The numerical scheme.} The mesh of the domain $\Omega$ is denoted by $\cal M=(\T,\E,\cal P)$ and classically given by:
$\mathcal T$, a set of open polygonal { or polyhedral} control volumes; 
$\mathcal E$, a set of edges {or faces}; ${\mathcal P}=(x_K)_{K\in\mathcal T}$ a set of points. {In the following, we also use the denomination ``edge'' for  a face in dimension $3$}. As we deal with a Two-Point Flux Approximation (TPFA) of convection-diffusion equations, we assume that the mesh is admissible in the sense of~\cite{Eymard2000} (Definition 9.1). 

We distinguish in $\E$ the interior edges, $\sigma =K|L$, from the exterior edges:  $\E=\E_{int}\cup {\mathcal E}_{ext}$. Among the exterior edges, we distinguish the edges included in $\Gamma^D$ from the edges included in $\Gamma^N$: ${\mathcal E}_{ext}={\mathcal E}^D\cup {\mathcal E}^N$. For a given control volume $K\in{\mathcal T}$, we define ${\mathcal E}_K$ the set of its edges, which is also split into ${\mathcal E}_K={\mathcal E}_{K,int}\cup{\mathcal E}_{K}^D\cup{\mathcal E}_{K}^N$. For each edge $\sigma\in\E$, we pick one cell in the non empty set $\{K:\sigma\in\E_K\}$ and denote it by $K_\sigma$. In the case of an interior edge $\sigma=K|L$, $K_{\sigma}$ is either $K$ or $L$. 

{ Let ${\rm d}(\cdot,\cdot)$ denote the Euclidean distance.} For all edges $\sigma\in{\mathcal E}$, we set ${\rm d}_{\sigma}={\rm d}(x_K,x_L)$ if $\sigma=K|L\in{\mathcal E}_{int}$ and ${\rm d}_{\sigma}={\rm d}(x_K,\sigma)$ if $\sigma\in{\mathcal E}_{ext}$ with $\sigma\in \E_K$ and  the transmissibility coefficient is defined by $\tau_{\sigma}={\rm m}(\sigma)/{\rm d}_{\sigma}$, for all $\sigma\in{\mathcal E}$. We also denote by ${\mathbf n}_{K,\sigma}$ the normal to $\sigma\in{\mathcal E}_K$ outward $K$.
We assume that the mesh satisfies the regularity constraint:
\begin{equation}\label{reg-mesh} 
\exists \xi >0 \mbox{ such that } {\rm d}(x_K,\sigma)\geq \xi \, {\rm d}_{\sigma},\quad \forall K\in\T, \forall \sigma\in\E_K.
\end{equation}
As a consequence, we obtain that 
\begin{equation}\label{inegvol}
\sum_{\sigma\in\E_K} {\rm m}(\sigma)\dsig\leq \ds\frac{p}{\xi} {\rm m} (K)\quad \forall K\in\T.
\end{equation}
 The size of the mesh is defined by $h=\max\{\mbox{diam }(K)\,:\,K\in\T\}$.
 
 Let us define 
 $$
 \begin{aligned}
 & f_K=\ds\frac{1}{{\rm m}(K)}\int_{K} f, \quad b_K=\ds\frac{1}{{\rm m}(K)}\int_{K} b \quad \forall K\in\T,\\
 &U_{K,\sigma}=\ds\frac{1}{{\rm m}(\sigma)}\int_{\sigma} {\mathbf U}\cdot {\mathbf n}_{K,\sigma},\quad \forall K\in\T,\ \forall\sigma \in {\mathcal E}_K,\\
 & v_\sigma^D=\ds\frac{1}{{\rm m}(\sigma)}\int_\sigma v^D,\quad \forall \sigma\in {\mathcal E}^D.
 \end{aligned}
 $$
Given a Lipschitz-continuous function on $\R$ which satisfies 
\begin{equation}\label{hyp_B}
B(0)=1,\quad\ B(s)>0\quad \mbox{ and }\quad B(s)-B(-s)=-s\quad\forall s\in\R,
\end{equation}
we consider the B-scheme defined by 
\begin{equation}\label{scheme}
\sum_{\sigma\in \E_K} {\mathcal F}_{K,\sigma}+ {\rm m}(K) b_K v_K= {\rm m}(K)f_K, \quad \forall K\in{\mathcal T},
\end{equation}
where the numerical fluxes are defined by 
\begin{equation}\label{numflux}
{\mathcal F}_{K,\sigma}=\left\{
\begin{aligned}
&0,\quad \forall K\in\T,\forall \sigma \in \E_K^N,\\
&\tau_{\sigma} \Bigl(B(-U_{K,\sigma}\dsig)v_K-B(U_{K,\sigma}\dsig)v_{K,\sigma}\Bigl),\quad\forall K\in\T, \forall \sigma\in\E_{K}\setminus \E_K^N,
\end{aligned}
\right.
\end{equation}
with the convention $v_{K,\sigma}=v_L$ if $\sigma =K|L$ and $v_{K,\sigma}=v_\sigma^D$ if $\sigma \in\E_K^D$. 
Let us recall that the upwind scheme corresponds to the case $B(s)=1+s^-$  ($s^-$ is the negative part of $s$, while $s^+$ is its positive part) and the Scharfetter-Gummel scheme to the case 
 $B(s)=s/(e^s-1)$. They both satisfy \eqref{hyp_B}. The centered scheme which corresponds to $B(s)=1-s/2$ does not satisfy the positivity assumption. It can however be used if $|U_{K,\sigma}|\dsig\leq 2$ for all $K\in\T$ and $\sigma \in\E_K$. Thanks to the hypotheses \eqref{hyp_B}, we notice that the numerical fluxes through the interior and Dirichlet boundary edges rewrite
 \begin{equation}\label{numflux2}
 {\mathcal F}_{K,\sigma}= \tau_\sigma  B(|U_{K,\sigma}|\dsig)(v_K-v_{K,\sigma})+ {\rm m}(\sigma) \left(U_{K,\sigma}^+ v_K - U_{K,\sigma}^-v_{K,\sigma}\right) .
 \end{equation}
 
\noindent
{\bf Main result.} The scheme \eqref{scheme}-\eqref{numflux} defines a linear system of equations ${\mathbb M} {\mathbf v}={\mathbf S}$ whose unknown is ${\mathbf v}=(v_K)_{K\in\T}$; It is well-known that ${\mathbb M}$ is an M-matrix, which ensures existence and uniqueness of a solution to the scheme. Moreover, we may notice that, if $v^D$ and $f$ are nonnegative functions, then ${\mathbf S}$ has nonnegative values and therefore $v_K\geq 0$ for all $K\in\T$. Our purpose is now to establish $L^{\infty}$ bounds on ${\mathbf v}$ as stated in Theorem \ref{mainthm}.

\begin{theorem}\label{mainthm}
Assume that ${\mathbf U}\in C({\bar \Om})^2$, $b\in \Linf(\Omega)$ with $b\geq 0$ {\em a.e.}, ${f\in L^\infty(\Omega)}$ and $v^D\in \Linf (\Gamma^D)$. There exists non-negative constants $\overline{M}$ (\emph{resp.} $\underline{M}$) depending only on $\Omega$, $\xi$, the function $B$, $\|{\bf U}\|_{L^\infty}$, $\|f^+\|_{L^\infty}$ and $\|(v^D)^+\|_{L^\infty}$ (\emph{resp.} $\|f^-\|_{L^\infty}$ and $\|(v^D)^-\|_{L^\infty}$) such that the solution ${\mathbf v}$ to the scheme \eqref{scheme}-\eqref{numflux} verifies 
\[
-\underline{M}\ \leq\ v_K\ \leq\ \overline{M}, \quad \forall K\in\T. 
\]
\end{theorem}

The rest of this paper is dedicated to the proof of Theorem~\ref{mainthm}. It relies on a De Giorgi iteration method (see \cite{Vasseur_lectnotes} and references therein). In Section~\ref{sec:particular}, we start by studying a particular case where the data is normalized. Then, we give the proof of the theorem in Section~\ref{sec:proof}.

Let us mention that from the bounds of Theorem~\ref{mainthm}, it is possible to establish global-in-time $L^\infty$ bounds for the corresponding evolution equation by using an entropy method (see \cite[Theorem 2.7]{chainais_2019_large}).
 
\section{Study of a particular case}\label{sec:particular}

In this section, we consider the particular case where the source $f$ is non-negative and the boundary condition $v^D$ is non-negative and bounded by $1$.

Let us start with some notations. Given $m\geq 1$, we denote the $m$-th truncation threshold by  \begin{equation}\label{eq:trunc}
   C_m=2(1-2^{-m})\,,                                                                                            \end{equation}
Then, we introduce the $m$-th energy
\begin{equation}\label{eq:energy}
E_m({\bf v})=\ds\sum_{\sigma\in\E_{int}\cup\E^D} \tau_\sigma \left[\log (1 +(v_{K,\sigma}-C_m)^+)-\log (1 +(v_{K}-C_m)^+)\right]^2.
\end{equation} 
When there is no ambiguity we write $E_m = E_m({\bf v})$. The first proposition is a fundamental estimate of the energy.

\begin{proposition}\label{prop:fund_ineq}
 Assume that $f_K\geq 0$ for all $K\in\T$ and $v_\sigma^D \in [0,1]$ for all $\sigma \in \E^D$, so that the solution ${\mathbf v}$ to \eqref{scheme}-\eqref{numflux} satisfies $v_K\geq 0$ for all $K\in\T$. Then one has for all $m\geq1$ that
\begin{equation}\label{majEm}
E_m\ \leq\ \frac{4p}{\beta_{\mathbf U}^2}\left(\Vert {\mathbf U}\Vert_{L^\infty}^2+{\Vert f\Vert_{L^\infty}}\right) \sum_{\substack{K\in\T\\v_{K}>C_m}}{{\rm m}(K)}\,.
\end{equation}
where $\beta_{\mathbf U} := \inf_{x\in[-\|\mathbf{U}\|_{L^\infty},\|\mathbf{U}\|_{L^\infty}]}B({\rm diam}(\Omega)\,x)$  (because of \eqref{hyp_B}, $\beta_{\mathbf U}\in(0,1]$).
\end{proposition}
\begin{proof}
  In order to shorten some expressions hereafter, let us introduce $w_K^m = v_K-C_m$ for all $K\in \T$ and $w_\sigma^{m,D}=v_{\sigma}^D-C_m$ for all $\sigma \in\E^D$. Let us note that we identify ${\mathbf w}^m=(w_K^m)_{K\in\T}$ and the associate piecewise constant function. Therefore, we can write 
 $$
  {\rm m} (\{{\mathbf w}^m>0\}) = \sum_{w_{K}^m>0}{{\rm m}(K)}.
  $$
  
 First, observe that $E_m$ is the discrete counterpart of 
$$
\int_\Omega \left\vert \nabla \log (1+w^m)\right\vert^2 {\mathbf 1}_{\{w^m>0\}}=\int_\Omega \nabla w^m\cdot\frac{\nabla w^m}{(1+w^m)^2}{\mathbf 1}_{\{w^m>0\}},\ 
\mbox{ with }w^m=v-C_m\,,
$$
 {where ${\mathbf 1}_{A}$ is the indicator function of $A$}. Let us define $\varphi :s\mapsto s/(1+s){\mathbf 1}_{\{s\geq 0\}}$, which satisfies  $\varphi'(s)=1/(1+s)^2{\mathbf 1}_{\{s\geq 0\}}$ and let us introduce $F_m$ another  discrete counterpart of the preceding quantity
$$
F_m=\ds\sum_{\sigma\in\E_{int}\cup\E^D} \tau_\sigma \left((w_{K,\sigma}^m)^+-(w_{K}^m)^+\right)\left(\varphi(w_{K,\sigma}^m)-\varphi(w_{K}^m) \right). 
$$
It is clear that $E_m\leq F_m$ for all $m\geq 1$, as for all $x,y\in\R$ we have
$$
\left(\log(1+x^+)-\log(1+y^+)\right)^2\leq (x^+-y^+)\left(\varphi(x)-\varphi(y)\right).
$$

Let us now multiply the scheme \eqref{scheme} by $\varphi(w_K^m)$ and sum over $K\in\T$. Due to the non-negativity of $b$ and ${\mathbf v}$, we obtain, after a discrete integration by parts,
$$
\ds\sum_{\sigma\in\E_{int}\cup\E^D}{\mathcal F}_{K,\sigma}(\varphi(w_K^m)-\varphi(w_{K,\sigma}^m))\leq \sum_{K\in\T} {\rm m}(K) f_K \varphi(w_K^m).
$$ 
Using that $\varphi$ is bounded by 1 and vanishes on $\R_-$, we deduce that 
\begin{equation}\label{inegdep}
\ds\sum_{\sigma\in\E_{int}\cup\E^D}{\mathcal F}_{K,\sigma}(\varphi(w_K^m)-\varphi(w_{K,\sigma}^m))\leq { \Vert f\Vert_{L^\infty}\,
{\rm m}(\{{\mathbf w}^m>0\})}.
\end{equation}

We focus now on the left-hand-side of \eqref{inegdep}.  Due to \eqref{numflux2} and the definition of $w_K^m$, we can rewrite $ {\mathcal F}_{K,\sigma}$ as 
$$
{\mathcal F}_{K,\sigma}= \tau_\sigma  B(|U_{K,\sigma}|\dsig)(w_K^m-w^m_{K,\sigma})+ {\rm m}(\sigma) \left(U_{K,\sigma}^+ (w_K^m+C_m) - U_{K,\sigma}^-(w_{K,\sigma}^m+C_m)\right) .
$$
Observe that since $\varphi$ is a non-decreasing function, one has
$$
(x-y)\left(\varphi(x)-\varphi(y)\right)\geq (x^+-y^+)(\varphi(x)-\varphi(y)),\quad \forall x,y\in\R.
$$
Therefore, using the definition of $\beta_{\mathbf U}$ we obtain that 
\begin{equation}\label{ineg1}
\ds\sum_{\sigma\in\E_{int}\cup\E^D}{\mathcal F}_{K,\sigma}(\varphi(w_K^m)-\varphi(w_{K,\sigma}^m))\geq \beta_{\mathbf U} F_m -G_m,
\end{equation}
with 
$$
G_m=-\sum_{\sigma\in\E_{int}\cup\E^D}{\rm m}(\sigma) \left(U_{K,\sigma}^+ (w_K^m+C_m) - U_{K,\sigma}^-(w_{K,\sigma}^m+C_m)\right) (\varphi(w_K^m)-\varphi(w_{K,\sigma}^m)).
$$
For an interior edge, $w_K^m$ and $w_{K,\sigma}^m$ play a symmetric role in the preceding sum. As $w_{\sigma}^{m,D}\leq 0$ for all $\sigma\in\E^D$ and $\varphi$ vanishes on $\R_-$, we can always assume that $w_K^m\geq w_{K,\sigma}^m$ and an edge has a contribution in the sum if at least $w_K^m> 0$. Then, under these assumptions one has
\begin{multline*}
-{\rm m}(\sigma) \left(U_{K,\sigma}^+ (w_K^m+C_m) - U_{K,\sigma}^-(w_{K,\sigma}^m+C_m)\right) (\varphi(w_K^m)-\varphi(w_{K,\sigma}^m))\\
\leq \Vert {\mathbf U}\Vert_{L^\infty} {\rm m}(\sigma) (w_{K,\sigma}^m+C_m)(\varphi(w_K^m)-\varphi(w_{K,\sigma}^m)).
\end{multline*}
But, $w_{K,\sigma}^m+C_m\leq 2(1+(w_{K,\sigma}^m)^+)$ and applying the definition of $\varphi$, we get
$$
\begin{array}{rcl}
(w_{K,\sigma}^m+C_m)(\varphi(w_K^m)-\varphi(w_{K,\sigma}^m))&\leq& 2\ds\frac{(w_K^m)^+-(w_{K,\sigma}^m)^+}{1+(w_{K}^m)^+}\\[1em]
&\leq&2\ds\frac{(w_K^m)^+-(w_{K,\sigma}^m)^+}{\sqrt{1+ (w_K^m)^+}\sqrt{1+ (w_{K,\sigma}^m)^+}}.
\end{array}
$$
Therefore,
$$
G_m\leq 2\Vert {\mathbf U}\Vert_{L^\infty} \sum_{\sigma\in\E_{int}\cup\E^D}{\rm m}(\sigma)\ds\frac{\vert(w_K^m)^+-(w_{K,\sigma}^m)^+\vert}{\sqrt{1+ (w_K^m)^+}\sqrt{1+ (w_{K,\sigma}^m)^+}}.
$$
We apply now Cauchy-Schwarz inequality in order to get 
\begin{equation}\label{ineg2}
G_m\leq 2\Vert {\mathbf U}\Vert_{L^\infty}(F_m)^{1/2} \left(\sum_{\sigma\in\E^{sp}}{\rm m}(\sigma)\dsig\right)^{1/2},
\end{equation}
where $\E^{sp}$ is the set of interior and Dirichlet boundary edges on which $(w_K^m)^+-(w_{K,\sigma}^m)^+\neq 0$. It appears that, due to \eqref{inegvol},
\begin{equation}\label{ineg3}
\sum_{\sigma\in\E^{sp}}{\rm m}(\sigma)\dsig\leq \ds\sum_{K\in\T; w_K^m>0}\left(\sum_{\sigma\in \E_{K,int}\cup \E_K^D} {\rm m}(\sigma)\dsig\right)\leq \ds\frac{p}{\xi} {\rm m} (\{{\mathbf w}^m>0\}).
\end{equation}
We deduce from \eqref{inegdep}, \eqref{ineg1}, \eqref{ineg2} and \eqref{ineg3} that 
$$
\beta_{\mathbf U} F_m\leq {2}\Vert {\mathbf U}\Vert_{L^\infty} (F_m)^{1/2}(\frac{p}{\xi}{\rm m} (\{{\mathbf w}^m>0\}))^{1/2}+{\Vert f\Vert_{L^\infty}{\rm m} (\{{\mathbf w}^m>0\})},
$$
which yields \eqref{majEm} using Young's inequality and  the bounds $E_m\leq F_m$ and $\beta_{\mathbf U}\leq1$.
\end{proof}
{Before stating the main result of the section, we need a technical lemma.
\begin{lemma}\label{lem:sequence}
Let $(u_n)_{n\in\N}$ be a sequence of non-negative real numbers and let $K, \rho>0$ and $\alpha >1$. Then if for all $n\in\N$
\[
u_{n+1}\,\leq\, K\,\rho^{n}\,u_{n}^\alpha\,,
\]
one has
\[
0\leq u_n\,\leq\, \left(u_0\,\rho^{\frac{1}{(\alpha-1)^2}}\,K^{\frac{1}{\alpha-1}}\right)^{\alpha^n}\,\rho^{-\frac{n(\alpha-1)+1}{(\alpha-1)^2}}\,K^{-\frac{1}{\alpha-1}}
\]
 for all $n\in\N$ and the bound is optimal. In particular, if $u_0\leq \rho^{-\frac{1}{(\alpha-1)^2}}\,K^{-\frac{1}{\alpha-1}}$, then $\lim u_n=0$.
\end{lemma}
\begin{proof}
 Just observe that the sequence $v_n = u_n\,\rho^{\frac{n(\alpha-1) + 1}{(\alpha-1)^2}}\,K^{\frac{1}{\alpha-1}}$ satisfies $0\leq v_{n+1}\leq v_{n}^\alpha$ for all $n\geq0$ which directly yields the result.
\end{proof}}

{ \begin{proposition}\label{mainprop}
Assume that $f_K\geq 0$ for all $K\in\T$ and $v_\sigma^D \in [0,1]$ for all $\sigma \in \E^D$, so that $v_K\geq 0$ for all $K\in\T$. Then, there exists $\eta>0$ depending only on $\Omega$, $p$ and $\xi$ such that one has the implication
\begin{equation}\label{resprop}
\ds E_1\leq\ \eta\ \frac{\beta_{\mathbf U}^4}{(\|\mathbf{U}\|_{L^\infty}^2+\|f\|_{L^\infty})^2}\quad\Rightarrow\quad (v_K\leq 2,\ \forall K\in\T)\,.
\end{equation}
\end{proposition}
}
\begin{proof}
The proof consists in establishing an induction property on $E_m$ which guarantees that if $E_1$ is small enough then $\lim E_m =0$. Then, as $\lim C_m=2$ 
and thanks to the discrete Poincar\'e inequality,  we deduce  that 
$$
\ds\sum_{K\in\T} {\rm m}(K) \left( \log (1 +(v_K-2)^+)\right)^2=0,
$$
which implies $v_K\leq 2$ for all $K\in\T$.

For establishing the induction, first observe that as $C_m=C_{m-1} +2^{-m+1}$, for any $q>0$ we have:
\begin{equation}\label{eq:nonlinbound}
{\mathbf 1}_{\{{\mathbf w}^m>0\}}\leq \frac{\left(\log (1+({\mathbf w}^{m-1})^+)\right)^q}{(\log (1+2^{-m+1}))^q}{\mathbf 1}_{\{{\mathbf w}^{m-1}>0\}},
\end{equation}
and thus
$$
{\rm m} (\{{\mathbf w}^m>0\})\leq \frac{1}{(\log (1+2^{-m+1}))^q}\ds\sum_{K\in\T} {\rm m}(K) \left( \log (1 +(w_K^{m-1})^+)\right)^q.
$$
We may choose for instance $q=3$ and apply a discrete Poincar\'e-Sobolev inequality (whose constant $C_{\Omega,p}$ depends only on $\Omega$ and $p$), which leads to 
\begin{equation}\label{majmes}
{\rm m} (\{{\mathbf w}^m>0\})\leq\frac{1}{(\log (1+2^{-m+1}))^3}\frac{C(\Omega)}{\xi^{3/2}} E_{m-1}^{3/2}.
\end{equation}
Noticing that for $x\in[0,1]$, $(\log(1+x))^3\geq (\log 2)^3 x^3$, we deduce from \eqref{majEm} and \eqref{majmes} that
$$
E_m\leq \frac{4}{\beta_{\mathbf U}^2}\left(\Vert {\mathbf U}\Vert_{L^\infty}^2+\Vert f\Vert_{L^\infty}\right)\frac{{\tilde C}_{\Omega,p}}{\xi^{3/2}}8^{m-1}E_{m-1}^{3/2}.
$$
Thus the sequence $(E_m)_{m\geq 0}$ satisfies the hypothesis of Lemma~\ref{lem:sequence} with $\alpha=3/2$ and $K$ proportional to $(\Vert {\mathbf U}\Vert_{L^\infty}^2+\Vert f\Vert_{L^\infty})/\beta_{\mathbf U}^2$.  We deduce the upper bound for $E_1$ under which $\lim E_m =0$. 
\end{proof} 

{\em Remark:}  The arguments developed in this section still hold, up to minor adaptation, for $f\in L^r(\Omega)$ with $r>p/2$.

\section{Proof of Theorem \ref{mainthm}}\label{sec:proof}

First observe that if one replaces the data $f$ and $v^D$ by either $f^+$ and $(v^D)^+$, or $f^-$ and $(v^D)^-$, in the scheme \eqref{scheme}-\eqref{numflux}, then the corresponding solutions, say respectively $\mathbf{P}=(P_K)_{K\in\T}$ and $\mathbf{N}=(N_K)_{K\in\T}$, are non-negative and such that ${\mathbf v}=\mathbf{P}-\mathbf{N}$ is the solution to \eqref{scheme}-\eqref{numflux} in the original framework.

From there let us show that there is ${\overline M}>V^D_+:=\max(\|(v^D)^+\|_{L^\infty},1)$ such that for all $K\in\T$ one has $0\leq P_K\leq {\overline M}$. The bound for $\mathbf{N}$, which is denoted by ${\underline M}$, can be obtained in the same way. 

Let $M>V^D_+$. First observe that $\mathbf{P}^M:=\mathbf{P}/M$ satisfies the scheme \eqref{scheme}-\eqref{numflux} where the source term and boundary data have been replaced by $f^+/M$ and $(v^D)^+/M$ respectively. Moreover, one can apply Proposition~\ref{prop:fund_ineq}, which yields
\begin{equation}\label{eq:ineqPMagain}
 E_1(\mathbf{P}^M)\leq \frac{4p}{\beta_{\mathbf U}^2}\left(\Vert {\mathbf U}\Vert_{L^\infty}^2 + \frac{\Vert f^+\Vert_{L^\infty}}{M}\right) {\rm m} (\{{\mathbf P}^M>1\})\,.
\end{equation}
Now observe that $\mathbf{P}\,=\,M\,\mathbf{P}^M\,=\,V^D_+\,\mathbf{P}^{V^D_+}$. Therefore, 
$$
\begin{array}{rcl}
 \ds E_1(\mathbf{P}^M)&\leq& \ds \frac{4p}{\beta_{\mathbf U}^2}\Big{(}\Vert {\mathbf U}\Vert_{L^\infty}^2\,{\rm m} (\{{\mathbf P}^{ V^D_+}>M /  V^D_+\}) + \frac{\Vert f^+\Vert_{L^\infty}}{M}{\rm m}(\Omega)\Big{)} \\[1em]
 &\leq&\ds\frac{4p}{\beta_{\mathbf U}^2}\Big{(}\Vert {\mathbf U}\Vert_{L^\infty}^2 \sum_{K\in\T}{\rm m}(K)\,\frac{\log(1+(P_K^{ V^D_+}-1)^+)^2}{\log(M / V^D_+)^2} + \frac{\Vert f^+\Vert_{L^\infty}}{M}{\rm m}(\Omega)\Big{)}\\[1em]
&\leq&\ds\frac{C_{\Omega,p}}{\xi\beta_{\mathbf U}^2}\Vert {\mathbf U}\Vert_{L^\infty}^2\, \frac{E_1(\mathbf{P}^{V^D_+})}{\log(M /  V^D_+)^2} + \frac{4p\,{\rm m}(\Omega)}{\beta_{\mathbf U}^2}\frac{\Vert f^+\Vert_{L^\infty}}{M}\,\,,
\end{array}
$$
where we used an argument similar to \eqref{eq:nonlinbound} in the second inequality and a discrete Poincar\'e inequality in the third one. Then, by using \eqref{eq:ineqPMagain} again we get 
$$
 E_1(\mathbf{P}^{V^D_+})\ \leq\ \frac{4\,p\,{\rm m}(\Omega)}{\beta_{\mathbf U}^2}\left(\Vert {\mathbf U}\Vert_{L^\infty}^2 + \frac{\Vert f^+\Vert_{L^\infty}}{V^D_+}\right)
$$
Therefore, the smallness condition of Proposition~\ref{mainprop} is satisfied by $E_1(\mathbf{P}^M)$ if 
\begin{multline}\label{boundM}
 \left[\Vert{\mathbf U}\Vert_{L^\infty}^2\left(\Vert {\mathbf U}\Vert_{L^\infty}^2 + \frac{\Vert f^+\Vert_{L^\infty}}{V^D_+}\right)+ \frac{\Vert f^+\Vert_{L^\infty}}{M}\log\left(\frac{M}{V^D_+}\right)^2\right]\,\left(\Vert {\mathbf U}\Vert_{L^\infty}^2 + \frac{\Vert f^+\Vert_{L^\infty}}{M}\right)^2\\\ \leq\ \,C_{\Omega,\xi,p}\,\beta_{\mathbf U}^4\,\log\left(\frac{M}{V^D_+}\right)^2\,.
\end{multline}
It is clear that \eqref{boundM} is satisfied for $M$ large enough, which permits to define ${\overline M}$. Observe that if $v^D_+=0$ ($V_+^D=1$) and $\mathbf U=0$, ${\overline M} = \widetilde{C}_{\Omega,\xi,p}\|f^+\|_{L^\infty}$ works as expected.

 \smallskip
 
 \textbf{Acknowledgements.}  The authors thank the Labex CEMPI (ANR-11-LABX-0007-01) and the ANR MOHYCON (ANR-17-CE40-0027-01) for their support. They also want to thank Alexis F. Vasseur for fruitful exchanges on the subject.
 
\bibliographystyle{spmpsci}
\bibliography{bib-dGdiscret}

 \end{document}